\documentclass[11pt, reqno]{amsart}

\usepackage{lineno,hyperref}
\usepackage{cases}
\usepackage[square, comma, sort&compress, numbers]{natbib}%
\usepackage{float}
\usepackage{pifont}
\usepackage{bbding}
\usepackage{amssymb}
\usepackage{mathrsfs}
\usepackage{subfigure}
\usepackage{caption}
\usepackage{geometry}
\usepackage{color}
\usepackage{amsmath, amsthm, amscd, amsfonts, amssymb, graphicx, color}
\usepackage{lineno}
\usepackage{graphicx}
\usepackage{epsfig}

\usepackage{subfigure}
\newtheorem{thm}{Theorem}
\newtheorem{lemma}{Lemma}
\newtheorem{prop}{Proposition}

\newtheorem{cor}{Corollary}

\newtheorem{dfn}{Definition}
\newtheorem{rem}{Remark}

\numberwithin{equation}{section}
\numberwithin{equation}{section}
\numberwithin{dfn}{section}
\numberwithin{lemma}{section}
\numberwithin{theorem}{section}
\numberwithin{thm}{section}
\numberwithin{cor}{section}
\numberwithin{prop}{section}
\numberwithin{rem}{section}

\newcommand{\abs}[1]{\left\vert#1\right\vert}

\newcommand{\N}{\mbox{$\mathbb{N}$}}

\newcommand{\D}{\mbox{$\mathbb{D}$}}

\date{\today}
\begin{document}
\setcounter{page}{1}

\title[Some properties of certain close-to-convex harmonic mappings]
{Some properties of certain close-to-convex harmonic mappings}

\author[X.-Y. Wang, Z.-G. Wang, J.-H. Fan and  Z.-Y. Hu]{Xiao-Yuan Wang, Zhi-Gang Wang$^{*}$, Jin-Hua Fan and  Zhen-Yong Hu}

\vskip.10in
\address{\noindent Xiao-Yuan Wang \vskip.05in
	School of Science, Nanjing University of Science and Technology,
	Nanjing 210094, Jiangsu, P. R.
	China.}

\email{\textcolor[rgb]{0.00,0.00,0.84}{mewangxiaoyuan$@$163.com}}

\address{\noindent Zhi-Gang Wang\vskip.05in
School of Mathematics and Statistics, Hunan
	First Normal University, Changsha 410205, Hunan, P. R. China.}
\vskip.05in
\email{\textcolor[rgb]{0.00,0.00,0.84}{wangmath$@$163.com}}

\address{\noindent Jin-Hua Fan\vskip.05in
     School of Science, Nanjing University of Science and Technology,
	Nanjing 210094, Jiangsu, P. R. China.}\vskip.05in
\email{\textcolor[rgb]{0.00,0.00,0.84}{jinhuafan$@$hotmail.com}}

\address{\noindent Zhen-Yong Hu\vskip.05in
     School of Science, Nanjing University of Science and Technology,
	Nanjing 210094, Jiangsu, P. R. China.}\vskip.05in
\email{\textcolor[rgb]{0.00,0.00,0.84}{huzhenyongad$@$163.com}}
\thanks{$^*$Corresponding author.}



\begin{abstract}
In this paper,
 we determine the sharp estimates for Toeplitz determinants of a subclass of close-to-convex harmonic mappings.
Moreover, we obtain an improved version of Bohr's inequalities for a subclass of close-to-convex harmonic mappings,
 whose analytic parts are Ma-Minda convex functions.
\end{abstract}

\vskip.20in

\maketitle


\section{\large\bf {\sc Introduction}}\label{sec1}%
A complex-valued function $f$ in the unit disk $\mathbb{D}=\{z:|z|<1\}$ is called a harmonic mapping
if $\Delta f=4 f_{z \overline{z}}=0$.
Let $\mathcal{H}$ denote the class of sense-preserving harmonic mappings $f=h +\overline{g}$ in $\mathbb{D}$, where
\begin{equation}\label{eq-1.01}
 \  h(z)=z+\sum_{n=2}^{\infty} a_{n} z^{n}\  \text{and}\   g(z)=\sum_{n=1}^{\infty} b_{n} z^{n}
\end{equation}
 are $analytic$ functions in $\mathbb{D}$.
 Let $\mathcal{S}_{\mathcal{H}}$  be the subclass of $\mathcal{H}$ consisting of univalent mappings.
We observe that $\mathcal{S}_\mathcal{H}$ reduces to the class $\mathcal{S}$ of normalized univalent analytic functions,
 if the co-analytic part $g\equiv0$. Denote by $\mathcal{K}_{\mathcal{H}}$ the close-to-convex subclass of $\mathcal{S}_{\mathcal{H}}$.
 If $b_1 = 0$, then the class $\mathcal{K}_{\mathcal{H}}$ reduces to $\mathcal{K}^{0}_{\mathcal{H}}$.

  Lewy \cite{Lewy-1936} proved that $f=h+\overline{g}$ is locally univalent
 in $\mathbb{D}$ if and only if the Jacobian  $J_{f}=\left|h^{\prime}\right|^{2}-\left|g^{\prime}\right|^{2}\neq 0$ in $\mathbb{D}$.  Noting that  the harmonic
 mapping $f$   is sense-preserving, i.e. $J_{f}>0$ or
  $\left|h^{\prime}\right|>\left|g^{\prime}\right|$ in $\mathbb{D}$.
 At this point, its dilatation $\omega_{f}=g^{\prime} / h^{\prime}$ has the property
$\left|\omega_{f}\right|<1$ in $\mathbb{D}$.  The reader can find much information about planar harmonic
 mappings from \cite{Clunie-Sheil-Small-1984,Duren-2004,Ponnusamy-Rasila-2013}.

Let $\mathcal{P}$ denote the class of analytic functions $p$ in $\mathbb{D}$ of the form
\begin{equation}\label{eq-2.00}
p(z)=1+\sum_{n=1}^{\infty} p_{n} z^{n}
\end{equation}
such that $\operatorname{Re} (p(z))>0$ in $\mathbb{D} .$

Denote by $\mathcal{A}$  the class of analytic functions in $\mathbb{D}$  with $f(0)=f'(0)-1=0$, and $\mathcal{K}(\alpha)$ denotes the class of functions $f \in \mathcal{A}$ such that
\begin{equation}\label{eq-1.01+1}
{\rm Re}\left(1+\frac{z f^{\prime \prime}(z)}{f^{\prime}(z)}\right)>\alpha \quad \left(-\frac{1}{2}\leq\alpha<1;\ z \in \mathbb{D} \right).
\end{equation}
Particularly, the elements in $\mathcal{K}(-1 / 2)$ are close-to-convex but are not necessarily
starlike in $\mathbb{D}$. For $0 \leq \alpha<1$, the elements in $\mathcal{K}(\alpha)$ are known to be convex functions of order $\alpha$ in $\mathbb{D}$.
 For more properties of starlike and convex functions, the reader can refer to the
monographs \cite{Duren-1983,Thomas-Tuneski-Vasudevarao-2018}.

 By making use of the subordination in analytic functions, Ma and Minda  \cite{Ma-Minda-1992}   introduced a more general class
$\mathcal{C}(\phi)$, consisting of functions in $\mathcal{S}$ for which $$1+ \frac{zf''(z)}{f'(z)} \prec \phi (z).$$ Here the function $\phi :\mathbb{D} \rightarrow \mathbb{C}$,
called Ma-Minda function, is analytic and univalent in $\mathbb{D}$ such that $\phi(\mathbb{D})$ has
positive real part, symmetric with respect to the real axis, starlike with respect to $\phi(0)=1$
and $\phi'(0)>0$ (for more details, see \cite{Wani-Swaminathan-2021,Sharma-Raina-2019}). A Ma-Minda function has  the form $$\phi(z)=1+ \sum_{n=1}^{\infty} B_{n}z^{n}.$$
The extremal function $K$  for the class $\mathcal{C}(\phi)$  is given by
\begin{equation} \label{eq-1.02}
K(z)=\int_{0}^{z}\exp \left(\int_{0}^{\zeta} \frac{\phi(t)-1}{t} d t\right) d \zeta \quad (z \in \mathbb{D}),
\end{equation}
which satisfies the condition $$1+\frac{zK''(z)}{K'(z)}=\phi (z).$$

We recall the following natural class
of close-to-convex harmonic mappings   $\mathcal{M}(\alpha,\zeta,n)$,  due to Wang {\it et\ al.} \cite{Wang-Liu-2018} (see also \cite{Ponnusamy-Kaliraj-2015, whlk}).

\begin{dfn}\label{d1}
	{\rm A harmonic mapping $f=h+\overline{g}\in\mathcal{H}$ is said to be in the class $\mathcal{M}(\alpha,\zeta,n)$ if
		$h \in \mathcal{K}(\alpha)$,  for some $\alpha\in \left[-{1}/{2},1\right)$, given by \eqref{eq-1.01+1} and $g$ satisfies the condition
		\begin{equation}\label{eq-2.04}
			g'(z)=\zeta z^nh'(z)
			\quad \left(\zeta\in\mathbb{C}\ {\rm with}\ \abs{\zeta}\leq \frac{1}{2n-1};\, n\in\N:=\{1,2,3,\cdots\}\right).
	\end{equation}}
\end{dfn}

For $n=1$, $\alpha=-{1}/{2}$ and $|\zeta|=1$, the class $\mathcal{M}(-{1}/{2},\zeta,1)$ was introduced by Bharanedhar and Ponnusamy \cite{Bharanedhar-Ponnusamy-2014}. For $n=1$, the class  $\mathcal{M}(\alpha,\zeta,1)$  was studied in \cite{Allu-Halder-2020-2,Sun-Jiang-Rasila-2016}.

In 2020, Allu and Halder \cite{Allu-Halder-2020-2} introduced and investigated the following subclass $\mathcal{H C}(\phi)$
of close-to-convex harmonic mappings.

\begin{dfn}\label{d2}
	{\rm For $\zeta \in \mathbb{C}$ with $|\zeta| \leq 1$, let $\mathcal{H C}(\phi)$  denote the class of harmonic mappings $f=h+\overline{g}$ in $\mathbb{D}$ of the form \eqref{eq-1.01}, whose analytic part $h$ belongs to $\mathcal{C}(\phi)$ and $h^{\prime}(0) \neq 0$, along with the condition $g^{\prime}(z)=\zeta z h^{\prime}(z)$}.
\end{dfn}

Motivated essentially by the classes $\mathcal{M}(\alpha,\zeta,n)$ and $\mathcal{H C}(\phi)$, we define a new subclass $\mathcal{HC}_n(\phi)$ of close-to-convex harmonic mappings as follows:
\begin{dfn}\label{d3}
	{\rm A harmonic mapping $f=h+\overline{g}\in\mathcal{H}$ is said to be in the class $\mathcal{HC}_n(\phi)$ if
	$h\in \mathcal{C}(\phi)$ and $g$ satisfies the condition \eqref{eq-2.04}.}
\end{dfn}

In 2019, Sun $et\ al.$ \cite{Sun-Wang-2019} investigated upper bounds of the third Hankel determinants for  the class $\mathcal{M}(\alpha,1,1)$ of
close-to-convex harmonic mappings.
In recent years, the Toeplitz determinants and Hankel determinants of functions in the
class $\mathcal{S}$ or its subclasses  have attracted many
researchers' attention (see \cite{Babalola-2010,Cudna-Kwon-2019,Kowalczyk-Kwon-2019,Jastrzebski-Kowalczyk-2020,Cho-Kumar-2020,
Dobosz-2021,Lecko-miarowska-2021,Kumar-2021,Lecko-Sim-2020,Janteng-Halim-Darus-2007}).
Among them,  the symmetric Toeplitz determinant $|T_q(n)|$ for subclasses of $\mathcal{S}$ with small  values of $n$
and $q$,  are investigated by \cite{Arif-Raza-2019,Ali-Thomas-Vasudevarao-2018,Radhika-Sivasubramanian-Murugusundaramoorthy-Jahangiri-2016,Ahuja-Khatter-2021,
Wang-Huang-Long-2021,Zhang-Srivastava-2019}.

The symmetric Toeplitz determinant $T_q(n)$  for  analytic functions $f$ is defined
  as follows:
$$
T_{q}(n)[f]:=
\begin{vmatrix}
	a_n & a_{n+1} & \cdots & a_{n+q-1}\\
	a_{n+1} & a_n & \cdots & a_{n+q-2}\\
	\vdots & \vdots & \vdots & \vdots &\\
	a_{n+q-1} & a_{n+q-2} & \cdots & a_{n}
\end{vmatrix},
$$
where $n,q\in\N$ and $a_1=1$. In particular, for functions in starlike and convex classes,
$T_{2}(2)[f],\, T_{3}(1)[f]$ and $T_{3}(2)[f]$ were studied by Ali $et\ al.$ \cite{Ali-Thomas-Vasudevarao-2018}.

Let $ \mathcal{B} $ be the class of analytic functions $ f $ in $\mathbb{D}$ such that $ |f(z)|<1 $
for all $ z\in\mathbb{D} $, and let $ \mathcal{B}_0=\{f\in\mathcal{B} : f(0)=0\} $. In $ 1914 $,
Bohr \cite{Bohr-1914} proved that if $ f\in\mathcal{B} $ is of the form $ f(z)=\sum_{n=0}^{\infty}a_nz^n $,
then the majorant series $ M_f(r)= \sum_{n=0}^{\infty}|a_n||z|^n $ of $ f $ satisfies
\begin{equation}\label{eq-1.03}
	M_{f_0}(r)=\sum_{n=1}^{\infty}|a_n||z|^n\leq 1-|a_0|=d(f(0),\partial f(\mathbb{D}))
\end{equation}
for all $z\in\mathbb{D}$ with $|z|=r\leq 1/3$, where $f_0(z)=f(z)-f(0)$.  Bohr actually
obtained the inequality \eqref{eq-1.03} for $ |z|\leq 1/6 $. Moreover, Wiener, Riesz and Schur,
independently, established the Bohr  inequality \eqref{eq-1.03}   for $ |z|\leq 1/3 $
(known as Bohr radius for the class $ \mathcal{B} $) and proved that $1/3$ is the best possible.

The Bohr phenomenon was reappeared  in the $ 1990 $s due to Dixon \cite{Dixon-1995}. Furthermore, Boas and Khavinson \cite{Boas--Khavinson-1997}
found
bounds for Bohr's radius in any complete Reinhard domains.
Other works one can see \cite{Paulsen-Singh-2004,Blasco-2010,Muhanna-2010,Aizenberg-2000,Aizenberg-Aytuna-Djakov-2000}.
In recent years, Bohr inequality and Bohr radius have become an active research field in geometric function theory
(see \cite{Allu-Halder-2021,Ali-Abdulhadi-Ng-2016,Kayumov-Ponnusamy-2020,Liu-2021,Muhanna-Ali-2014,Liu-Ponnusamy-2021,Ismagilov-Kayumov-Ponnusamy-2020}).
 Furthermore, initiated by the work of \cite{Kayumov-Ponnusamy-Shakirov-2018}, the Bohr's phenomenon for the complex-valued harmonic mappings have been widely studied
 (see \cite{Ahamed-Allu-2021, Allu-Halder-2020-2, Evdoridis-Ponnusamy-Rasila-2019,Huang-Liu-Ponnusamy-2021,
 Kayumov-Ponnusamy-2018,Liu-Ponnusamy-2019}).

In this paper,
 we aim at determining the sharp estimates for Toeplitz determinants of the class $\mathcal{M}(\alpha,\zeta,n)$.
Moreover, we will derive an improved version of Bohr's inequalities for the class $\mathcal{HC}_n(\phi)$.

\section{\large\bf {\sc Preliminary results}}\label{sec2}%
To prove our main results, we need the following lemmas.

\begin{lemma}\label{lem-2.010}
		{\rm(\cite[p. 41]{Duren-1983})}	For a function $p \in \mathcal{P}$ of the form \eqref{eq-2.00}, the sharp
	inequality $\left|p_{n}\right| \leq 2$ holds for each $n \geq 1 .$ Equality holds for the function $$p(z)=\frac{1+z}{1-z}.$$	
\end{lemma}

\begin{lemma}\label{lem-2.011}
	{\rm(\cite[Theorem 1]{Efraimidis-2016})}	Let $p \in \mathcal{P}$ be of the form \eqref{eq-2.00} and $\mu \in \mathbb{C}$. Then
$$
\left|p_{n}-\mu p_{k} p_{n-k}\right| \leq 2 \max \{1,\,|2 \mu-1|\} \quad (1 \leq k \leq n-1).
$$
If $|2 \mu-1| \geq 1$, then the inequality is sharp for the function $$p(z)=\frac{1+z}{1-z}$$ or its rotations. If $|2 \mu-1|<1$, then the inequality is sharp for $$p(z)=\frac{1+z^{n}}{1-z^{n}}$$ or its rotations.
\end{lemma}

\begin{lemma}\label{lem-2.01}{\rm(\cite{Wang-Liu-2018})}
Let $f=h+\overline{g}\in\mathcal{M}(\alpha,\zeta,n)$. Then
the coefficients $a_k\ (k\in\N\setminus\{1\})$ of $h$ satisfy
\begin{equation}\label{e+2.01}
	\abs{a_k}\leq\frac{1}{k!}\prod_{j=2}^{k}(j-2\alpha)\ (k\in\N\setminus\{1\}).
\end{equation}
Moreover, the coefficients $b_k\ (k=n+1, n+2, \cdots; n\in\N)$ of $g$ satisfy
\begin{equation}\label{e+2.02}
\abs{b_{n+1}}\leq\frac{\abs{\zeta}}{n+1}\  \text{ and }\   \abs{b_{k+n}}\leq\frac{\abs{\zeta}}{(k+n)(k-1)!}\prod_{j=2}^{k}(j-2\alpha)\ (k\in\N\setminus\{1\};\, n\in\N).
\end{equation}
The bounds are sharp for the extremal function given by
\begin{equation}\label{e+2.03}
	f(z)=\int_0^z\frac{dt}{(1-\delta t)^{2-2\alpha}}+\overline{\int_0^z\frac{\zeta t^n}{(1-\delta t)^{2-2\alpha}}dt}\quad(\abs{\delta}=1;\, z\in\D).
\end{equation}
\end{lemma}

\begin{lemma}\label{lem-2.02}{\rm(\cite{Wang-Liu-2018})}
Let $f\in\mathcal{M}(\alpha,\zeta,n)$ with $0\leq\alpha<1$ and $0\leq\zeta<\frac{1}{2n-1}\ (n\in\N)$.
Then
\begin{equation}\label{eq-2.05}
\Phi(r;\alpha,\zeta,n)\leq \abs{f(z)}\leq \Psi(r;\alpha,\zeta,n)\quad(r=|z|<1),
\end{equation}
where
\begin{equation*}
\Phi(r;\alpha,\zeta,n)
=\left\{\begin{array}{ll}
\log (1+r)-\displaystyle\frac{\zeta\,  r^{n+1} \, _2F_1(1,\, n+1;\, n+2;\, -r)}{n+1}  &(\alpha=1/2),\\ \\
\displaystyle\frac{(1+r)^{2 \alpha -1}-1}{2 \alpha -1}-\frac{\zeta\,  r^{n+1} \, _2F_1(n+1,\, 2-2 \alpha;\, n+2;\, -r)}{n+1}   &(\alpha\neq 1/2),
\end{array}\right.
\end{equation*}
and
\begin{equation*}
\Psi(r;\alpha,\zeta,n)
=\left\{\begin{array}{ll}-\log (1-r)+\displaystyle\frac{\zeta\,  r^{n+1} \, _2F_1(1,\, n+1;\, n+2;\, r)}{n+1}  & (\alpha=1/2), \\ \\
\displaystyle\frac{1-(1-r)^{2 \alpha -1}}{2 \alpha -1}+\frac{\zeta\,  r^{n+1} \, _2F_1(n+1,\, 2-2 \alpha;\, n+2;\, r)}{n+1}   & (\alpha\neq 1/2).
\end{array}\right.
\end{equation*}
All these bounds are sharp, the extremal function is $f_{\alpha,\zeta,n}=h_{\alpha}+\overline{g_{\alpha,\zeta,n}}$
or its rotations, where
\begin{equation}\label{eq-2.06}
f_{\alpha,\zeta,n}(z)
=\left\{\begin{array}{ll}
-\log(1-z)+\displaystyle\overline{\frac{\zeta\,  z^{n+1} \, _2F_1(1,\, n+1;\, n+2;\, z)}{n+1}}  &(\alpha=1/2),  \\ \\
\displaystyle\frac{1-(1-z)^{2\alpha-1}}{2\alpha-1}
+\overline{\frac{\zeta\,  z^{n+1} \, _2F_1(n+1,\, 2-2 \alpha;\, n+2;\, z)}{n+1}}   &(\alpha\neq 1/2).
\end{array}\right.
\end{equation}
\end{lemma}

 The following two results are due to Ma and Minda \cite{Ma-Minda-1992}.
\begin{lemma} \label{lem-2.04}
	Let $f \in \mathcal{C}(\phi)$. Then $zf''(z)/f'(z) \prec zK''(z)/K'(z)$ and $f'(z) \prec K'(z)$, where $K$ is given by \eqref{eq-1.02}.
\end{lemma}
\begin{lemma} \label{lem-2.05}
	Assume that $f \in \mathcal{C}(\phi)$ and $|z|=r<1$. Then
	\begin{equation} \label{eq-2.07}
	K'(-r) \leq |f'(z)| \leq K'(r),
	\end{equation}
where $K$ is given by \eqref{eq-1.02}.	Equality holds for some $z \neq 0$ if and only if $f$ is a rotation of $K$.
\end{lemma}

\begin{lemma} {\rm (\cite{Bhowmik-Das-2018})}\label{lem-2.07}
	Let $f(z)=\sum_{n=0}^{\infty} a_{n}z^{n}$ and $g(z)=\sum_{n=0}^{\infty} b_{n}z^{n}$ be two analytic functions in $\mathbb{D}$ and $g \prec f$. Then
	\begin{equation} \label{eq-2.08}
	\sum_{n=0}^{\infty} |b_{n}| r^{n} \leq \sum_{n=0}^{\infty} |a_{n}| r^{n}
	\end{equation}
	for $|z|=r \leq 1/3.$
\end{lemma}

\begin{rem}{\rm
Lemma \ref{lem-2.07} continues to hold for quasi-subordination (cf. \cite{Alkhaleefah-Kayumov-Ponnusamy-2019}). Moreover, the
bound 1/3 is optimal as shown by  \cite[Lemma 1]{Ponnusamy-Vijayakumar-Wirths-2022}.}
\end{rem}

\section{\large\bf {\sc Toeplitz determinants for the class $\mathcal{M}(\alpha,\zeta,n)$}}\label{sec3}%
In this section, we will give several sharp estimates for Toeplitz determinants $|T_{q}(n)[\cdot]|$ of
functions in the class $\mathcal{M}(\alpha,\zeta,n)$.
\begin{thm}\label{thm-5.01}
Let  $f\in \mathcal{M}(\alpha,\zeta,n)$. Then
\begin{equation}\label{eq-5.01}
|T_{2}(n)[h]|\leq\left(\frac{1}{n!}\prod_{j=2}^{n}(j-2\alpha)\right)^2+
\left(\frac{1}{(n+1)!}\prod_{j=2}^{n+1}(j-2\alpha)\right)^2\quad (n\in \mathbb{N}\backslash \{1\}),
\end{equation}
and
\begin{equation}\label{eq-5.02}
    |T_{2}(n)[g]|\leq \frac{1}{[(2n-1)(n+1)]^2}.
\end{equation}
 The inequalities in \eqref{eq-5.01} and \eqref{eq-5.02} are sharp for the extremal function given by \eqref{e+2.03}.
\end{thm}

\begin{proof}
   Suppose that $f\in \mathcal{M}(\alpha,\zeta,n)$. By Lemma \ref{lem-2.01}, we see that
    \begin{equation}\label{eq-5.02+1}
        |T_{2}(n)[h]|=|a_n^2-a_{n+1}^{2}|\leq |a_n^2| +|a_{n+1}^2|
    \end{equation}
    yields \eqref{eq-5.01}. Equality in \eqref{eq-5.02+1} holds for the function $h$ given by
    \begin{equation}\label{eq-5.02+2}
    \begin{aligned}
        h(z)& = \int_0^z\frac{dt}{(1-\delta t)^{2-2\alpha}}\\
            & = z+\frac{1}{2} (2-2 \alpha) \delta^2 z^2+\frac{1}{6} (2-2\alpha) (3-2 \alpha) \delta^3z^3
            +\frac{1}{24} (2-2 \alpha)(3-2 \alpha)(4-2 \alpha)\delta^4 z^4 \\
           &\qquad\qquad+\frac{1}{120} (2-2 \alpha)(3-2 \alpha)(4-2 \alpha)(5-2 \alpha)\delta^5 z^5+ \cdots \quad(\abs{\delta}=1;\, z\in\D).
    \end{aligned}
    \end{equation}
By virtue of \eqref{e+2.02}, we get the assertion
\eqref{eq-5.02}. The proof of Theorem \ref{thm-5.01} is thus completed.
\end{proof}

\begin{cor}\label{cor-5.01}
    Let $f\in \mathcal{M}(\alpha,\zeta,2)$. Then
    \begin{equation}\label{eq-5.03}
        |T_{2}(2)[h]|\leq \frac{2}{9} (1-\alpha)^2 \left(2 \alpha^2-6 \alpha+9\right),
    \end{equation}
        and
\begin{equation}\label{eq-5.04}
    |T_{2}(2)[g]|\leq \frac{1}{81}.
\end{equation}
The inequalities in \eqref{eq-5.03} and \eqref{eq-5.04} are sharp for the extremal function given by \eqref{e+2.03} with $n=2$.
\end{cor}

\begin{thm}\label{thm-5.02}
    Let $f\in \mathcal{M}(\alpha,\zeta,1)$. Then
    \begin{equation}\label{eq-5.05}
    \left|T_{3}(1)[h]\right|\leq \begin{cases}
        {\frac{1}{9} \big(8 \alpha^4-34 \alpha^3+71 \alpha^2-72 \alpha+36\big) \quad \ (-\frac{1}{2}\leq\alpha\leq \frac{1}{2}}),  \\ \\[-12pt]
        {\frac{1}{9} \left(-2 \alpha^3+25 \alpha^2-44 \alpha+30\right) \quad\quad \quad\quad\ (\frac{1}{2}\leq \alpha< 1}),
	\end{cases}
    \end{equation}
    and
    \begin{equation}\label{eq-5.06}
        |T_{3}(1)[g]|\leq \frac{1}{3}(1-\alpha).
    \end{equation}
The inequality in \eqref{eq-5.05} is sharp for the function $h$ given by \eqref{eq-5.02+2}, and the inequality in \eqref{eq-5.06} is sharp for
the function $g$ defined by
\begin{equation}\label{eq-5.10+2}
g(z)={\int_0^z\frac{\zeta t}{(1-\delta t)^{2-2\alpha}}dt} \quad(\abs{\delta}=1;\, \abs{\zeta}\leq 1;\, z\in\D).
\end{equation}
\end{thm}

\begin{proof}
	For $f\in \mathcal{M}(\alpha,\zeta,1)$,   we see that
$$
p(z)=\frac{1}{1-\alpha}\left(1+\frac{z h^{\prime \prime}(z)}{h^{\prime}(z)}-\alpha\right)\in \mathcal{P} \quad\left(-\frac{1}{2}\leq \alpha<1;\, z \in \mathbb{D}\right).
$$
It follows that
\begin{equation}\label{eq-5.07}
n(n-1) a_{n}=(1-\alpha) \sum_{k=1}^{n-1} k a_{k} p_{n-k} \quad(n \geq 2).
\end{equation}
From  \eqref{eq-5.07}, we obtain
\begin{equation}\label{eq-5.08}
	\begin{cases}
a_{2}=\frac{1}{2}(1-\alpha) p_{1}, \\ \\[-12pt]
a_{3}=\frac{1}{6}(1-\alpha)\left[(1-\alpha) p_{1}^{2}+p_{2}\right], \\  \\[-12pt]
a_{4}=\frac{1}{24}(1-\alpha)\left[(1-\alpha)^{2} p_{1}^{3}+3(1-\alpha) p_{1} p_{2}+2 p_{3}\right]. \\
\end{cases}
\end{equation}
By virtue of Lemma \ref{lem-2.011} and \eqref{eq-5.08}, we get
 \begin{equation}\label{eq-5.08+1}
   \begin{aligned}
\left|T_{3}(1)[h]\right| &=\left|1-2 a_{2}^{2}+2 a_{2}^{2} a_{3}-a_{3}^{2}\right| \\
& \leq 1+2\left|a_{2}^{2}\right|+\left|a_{3}\right|\left|a_{3}-2 a_{2}^{2}\right| \\
& \leq 1+ \frac{1}{2}(1-\alpha)^2p_1^2+ \frac{1}{36} (1-\alpha)^2|(1-\alpha)p_1^2+p_2||p_2-2(1-\alpha)p_1^2| \\
&\leq \begin{cases}
      {\frac{1}{9} \big(8 \alpha^4-34 \alpha^3+71 \alpha^2-72 \alpha+36\big) \quad  (-\frac{1}{2}\leq\alpha\leq \frac{1}{2}}),  \\  \\[-12pt]
      {\frac{1}{9} \big(-2 \alpha^3+25 \alpha^2-44 \alpha+30\big)\qquad\quad\ \   (\frac{1}{2}\leq \alpha< 1}).
       \end{cases}
   \end{aligned}
 \end{equation}

In what follows, we shall prove that the equality in \eqref{eq-5.08+1} holds for the function $h$ given by \eqref{eq-5.02+2}.
It follows from \eqref{eq-5.02+2} that

 \begin{equation}\label{eq-5.08+2}
	\begin{cases}
|a_{2}|= 1-\alpha, \\ \\[-12pt]
|a_{3}|=\frac{1}{3}(1-\alpha)(3-2\alpha). \\  \\[-12pt]
\end{cases}
\end{equation}
 Therefore, we obtain
 \begin{equation}\label{eq-5.08+3}
   \begin{aligned}
\left|T_{3}(1)[h]\right| &=\left|1-2 a_{2}^{2}+2 a_{2}^{2} a_{3}-a_{3}^{2}\right| \\
& \leq 1+2\left|a_{2}^{2}\right|+\left|a_{3}\right|\left|a_{3}-2 a_{2}^{2}\right| \\
& = 1+ 2(1-\alpha)^2 + \frac{1}{3} (1-\alpha)  (3-2\alpha) \left| \frac{1}{3} (1-\alpha)  (3-2\alpha)-2(1-\alpha)^2 \right| \\
&=\begin{cases}
      {\frac{1}{9} \big(8 \alpha^4-34 \alpha^3+71 \alpha^2-72 \alpha+36\big) \quad\ \,  (-\frac{1}{2}\leq\alpha\leq \frac{1}{2}}),  \\  \\[-12pt]
      {\frac{1}{9} \big(-2 \alpha^3+25 \alpha^2-44 \alpha+30\big)\qquad\quad\ \   (\frac{1}{2}\leq \alpha< 1}).
       \end{cases}
   \end{aligned}
 \end{equation}

\begin{figure}[!htb]
	\begin{center}
	\includegraphics[width=0.45\linewidth]{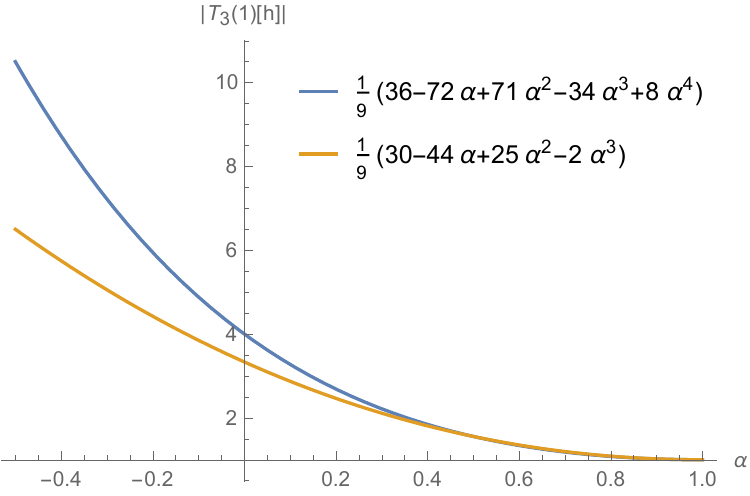}
	\includegraphics[width=0.45\linewidth]{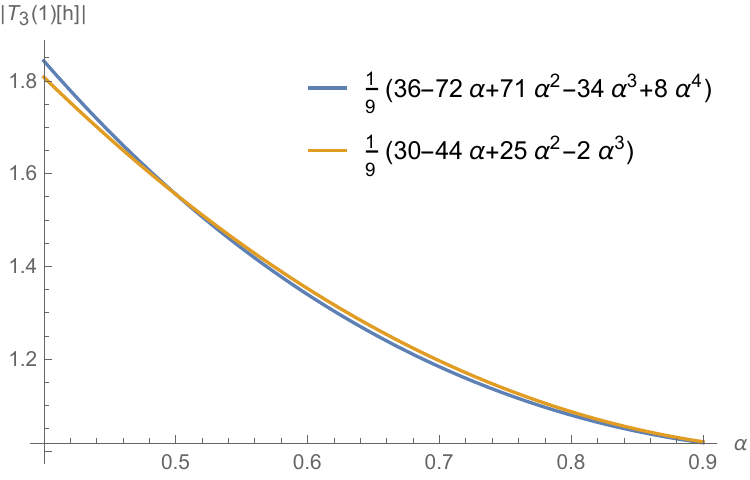}
	\end{center}
	\caption{The left graph is $|T_{3}(1)[h]|$ of $\frac{1}{9} \big(8 \alpha^4-34 \alpha^3+71 \alpha^2-72 \alpha+36\big)$ and $\frac{1}{9} \big(-2 \alpha^3+25 \alpha^2-44 \alpha+30\big)$, respectively;  The right graph  is a locally enlarged version.}
	\label{figure-2.1}
\end{figure}

By the power series representations of $h$ and $g$ for $f=h+\overline{g} \in \mathcal{M}(\alpha,\zeta,1)$, we find that
$$
(k+1) b_{k+1}=\zeta k a_{k} \quad (k\in \mathbb{N};\, |\zeta|\leq 1;\, a_1=1),
$$
which implies that
\begin{equation}\label{eq-5.10}
	\begin{cases}
b_{2} = \frac{1}{2}\zeta a_{1},\\   \\[-12pt]
b_{3} = \frac{2}{3}\zeta a_{2}.\\ \\[-12pt]
\end{cases}
\end{equation}
Thus, by Lemma \ref{lem-2.010}, \eqref{eq-5.08} and \eqref{eq-5.10}, we deduce that the assertion \eqref{eq-5.06} of Theorem \ref{thm-5.02} is true.
The sharpness of \eqref{eq-5.06} follows from \eqref{eq-5.05}.
\end{proof}

\begin{thm}\label{thm-5.03}
    Let $f\in \mathcal{M}(\alpha,\zeta,2)$. Then
    \begin{equation}\label{eq-5.11}
    |T_{3}(2)[h]|\leq
	\begin{cases}
        {\frac{1}{108} \left(1-\alpha)^3  (2 \alpha^2-7 \alpha+12 )  (10 \alpha^2-27 \alpha+36 \right) \quad  (-\frac{1}{2}\leq \alpha\leq \frac{1}{7}}), \\ \\[-12pt]
        {\frac{5}{108}  \left(1-\alpha)^3  (2 \alpha^2-7 \alpha+12 )  (2 \alpha^2-4 \alpha+7 \right) \qquad\ \,  (\frac{1}{7}\leq \alpha< 1}),
	\end{cases}
    \end{equation}
    and
    \begin{equation}\label{eq-5.12}
        |T_{3}(2)[g]|=|2b_3^2b_4|\leq \frac{1}{243}\left(1-\alpha\right).
    \end{equation}
    The inequality in \eqref{eq-5.11} is sharp for the function $h$ given by \eqref{eq-5.02+2}, and
     the  inequality in \eqref{eq-5.12} is sharp for the function $g$ defined by
\begin{equation}\label{eq-5.17}
g(z)={\int_0^z\frac{\zeta t^2}{(1-\delta t)^{2-2\alpha}}dt} \quad\left(\abs{\delta}=1;\, \abs{\zeta}\leq \frac{1}{3};\, z\in\D\right).
\end{equation}.
\end{thm}

\begin{proof}
Suppose that $f\in \mathcal{M}(\alpha,\zeta,2)$. It follows that $$T_{3}(2)[h]=\left(a_{2}-a_{4}\right)\left(a_{2}^{2}-2 a_{3}^{2}+a_{2} a_{4}\right).$$
In view of \eqref{eq-5.08} and Lemma \ref{lem-2.010}, we find that
    \begin{equation}\label{eq-5.13}
        \begin{aligned}
       \left|a_{2}-a_{4}\right| & \leq\left|a_{2}\right|+\left|a_{4}\right| \\
    &\leq \left|\frac{1}{2}(1-\alpha) p_{1}\right|+ \left|\frac{1}{24}(1-\alpha)\left[(1-\alpha)^{2} p_{1}^{3}+3(1-\alpha) p_{1} p_{2}+2 p_{3}\right]\right|\\
    &\leq \frac{1}{6}\left(1-\alpha)(2\alpha^2-7\alpha +12\right).
        \end{aligned}
    \end{equation}
    Next, we  shall  maximize $\left|a_{2}^{2}-2 a_{3}^{2}+a_{2} a_{4}\right|$.
    With the help of \eqref{eq-5.08},  Lemma \ref{lem-2.010} and  Lemma \ref{lem-2.011}, we get
    \begin{equation}\label{eq-5.14}
        \begin{aligned}
             |a_{2}^{2}-2 a_{3}^{2}+a_{2} a_{4} | &  =\frac{(1-\alpha)^2}{144} |-5(1-\alpha)^2p_1^4+36p_1^2-7(1-\alpha)p_1^2p_2-8p_2^2+6p_1p_3|\\
    &\leq \frac{(1-\alpha)^2}{144} \left[5(1-\alpha)^2|p_1|^4+36|p_1|^2+8|p_2|^2+6|p_1|\abs{p_3-\frac{7}{6}(1-\alpha)p_1p_2}\right]\\
    &\leq  \begin{cases}
        {\frac{1}{18}\big(1-\alpha)^2 \big(10 \alpha^2-27 \alpha+36\big)\quad\ \, (-\frac{1}{2}\leq\alpha\leq \frac{1}{7}}),\\  \\[-12pt]
        {\frac{5}{18} \big(1-\alpha)^2 \big(2 \alpha^2-4 \alpha+7\big) \quad\qquad (\frac{1}{7}\leq \alpha< 1}).
           \end{cases}
        \end{aligned}
    \end{equation}
Therefore, combining  \eqref{eq-5.13} with \eqref{eq-5.14}, we obtain \eqref{eq-5.11}.
By noting that for $f\in \mathcal{M}(\alpha,\zeta,2)$,  we have
\begin{equation}\label{eq-5.10+1}
	\begin{cases}
b_{3} = \frac{1}{3}\zeta a_{1},\\ \\[-12pt]
b_{4} = \frac{1}{2}\zeta a_{2}.\\
\end{cases}
\end{equation} By means of Lemma \ref{lem-2.010}, we get the assertion \eqref{eq-5.12}.
The sharpness of \eqref{eq-5.11} and \eqref{eq-5.12} are similar to that of Theorem \ref{thm-5.02}, we choose to omit the details here.
\end{proof}

\begin{rem}{\rm
By setting $\alpha=0$ in Corollary \ref{cor-5.01}, Theorem \ref{thm-5.02} and Theorem \ref{thm-5.03}, respectively, we get $|T_{2}(2)[h]|\leq 2$, $|T_{3}(1)[h]|\leq 4$ and $|T_{3}(2)[h]|\leq 4$. The bounds for convex functions were recently proved by
    Ali $et\ al.$ \cite[Theorem 2.11]{Ali-Thomas-Vasudevarao-2018}.}
\end{rem}

\section{\large\bf {\sc Bohr inequality for the class $\mathcal{HC}_n(\phi)$}}\label{sec4}
In this section, we firstly give
  the sharp growth estimate for the class $\mathcal{HC}_n(\phi)$.
\begin{prop} \label{prop-4.01}
Let $f \in \mathcal{HC}_n(\phi)$. Then
\begin{equation} \label{eq-3.01}
L(\zeta,n,r) \leq |f(z)| \leq R(\zeta,n,r),
\end{equation}
where
\begin{equation} \label{41}
L(\zeta,n,r)=-K(-r)-|\zeta|\int_{0}^{r} t^n K'(-t)dt,\end{equation}  {and} \begin{equation}\label{42} R(\zeta,n,r)=K(r)+|\zeta|\int_{0}^{r}t^n K'(t)dt.
\end{equation}
The bounds are sharp for the extremal function $f_{\zeta}=h_{\zeta}+\overline{g_{\zeta}}$ with $h_{\zeta}=K$, where $K$ satisfies \eqref{eq-1.02}
or its rotations and $g_{\zeta}$ satisfies $g'_{\zeta}=\zeta z^n h'_{\zeta}$.
\end{prop}

\begin{proof}
	Let $f=h+\overline{g} \in \mathcal{HC}_n(\phi)$. By Lemma \ref{lem-2.05}, we know  that
	\begin{equation} \label{eq-4.01}
		K'(-r) \leq |h'(z)| \leq K'(r) \quad (|z|=r).
	\end{equation}
	Let $\gamma$ be the linear segment joining $0$ to $z$ in $\mathbb{D}$. Then, we see that
	\begin{align} \label{eq-4.02}
		|f(z)|  =\left|\int_{\gamma}\frac{\partial f}{\partial \theta}\,\, d\theta +\frac{\partial f}{\partial \overline{\theta}}\,\, d\overline{\theta} \right| & \leq \int_{\gamma} \left(|h'(\theta)|+|g'(\theta)|\right)\, |d\theta|=\int_{\gamma} \left(1+|\zeta||\theta|^n\right)|h'(\theta)|\, |d\theta|.\\  \nonumber
	\end{align}
	Combining \eqref{eq-4.01} and \eqref{eq-4.02}, we obtain
	\begin{equation} \label{eq-4.03}
		|f(z)|\leq \int_{0}^{r} \left(1+|\zeta|t^n\right)K'(t) \, dt=K(r)+|\zeta|\int_{0}^{r}t^nK'(t)\,\,dt= R(\zeta,n,r).
	\end{equation}
	Let $\Gamma$ be the preimage of the line segment joining $0$ to $f(z)$ under the function $f$. It follows that
	\begin{equation}\label{eq-4.04}
	\begin{aligned}
		|f(z)|=\left|\int_{\Gamma}\frac{\partial f}{\partial \theta}\, d\theta +\frac{\partial f}{\partial \overline{\theta}}\,\, d\overline{\theta} \right|& \geq \int_{\Gamma} \left(|h'(\theta)|-|g'(\theta)|\right)\, |d\theta| \\
	   &=\int_{\Gamma} \left(1-|\zeta||\theta|^n\right)|h'(\theta)|\, |d\theta|.
	\end{aligned}
\end{equation}
	From \eqref{eq-4.01} and \eqref{eq-4.04}, we  have
	\begin{equation} \label{eq-4.05}
		|f(z)|\geq \int_{0}^{r} \left(1-|\zeta|t^n\right)K'(-t) \, dt=-K(-r)-|\zeta|\int_{0}^{r}t^nK'(-t)\, dt=L(\zeta,n,r).
	\end{equation}
	In view of \eqref{eq-4.03} and \eqref{eq-4.05}, we deduce that
	\begin{equation} \label{eq-4.06}
		L(\zeta,n,r) \leq |f(z)| \leq R(\zeta,n,r).
	\end{equation}
	
	To show the sharpness, we consider the function $f_{\zeta}=h_{\zeta}+\overline{g_{\zeta}}$ with $h_{\zeta}=K$ or its rotations.
	It is easy to see that $h_{\zeta}=K \in \mathcal{C}(\phi)$ and $g_{\zeta}$ satisfies $g'_{\zeta}(z)=\zeta z^n h'_{\zeta}(z)$, which shows that $f_{\zeta}\in \mathcal{HC}_n(\phi)$. The equality holds on both sides of \eqref{eq-4.01} for suitable rotations of $K$. For $0\leq \zeta< 1/{(2n-1)}$, we see that $f_{\zeta}(r)=R(\zeta,n,r)$ and $f_{\zeta}(-r)=-L(\zeta,n,r)$. Hence $|f_{\zeta}(r)|=R(\zeta,n,r)$ and $|f_{\zeta}(-r)|=L(\zeta,n,r)$. This completes the proof of Proposition \ref{prop-4.01}.
\end{proof}

\begin{prop} \label{prop-4.02}
Let $f \in \mathcal{HC}_n(\phi)$ and $S_{r}$ be the area of the image $f(\mathbb{D}_{r})$
${\rm (}\mathbb{D}_{r}:=\{z\in \mathbb{C}:|z|< r\leq1\}{\rm)}$.
Then
\begin{equation} \label{eq-3.02}
2\pi \int_{0}^{r}t\left(1-|\zeta|^{2}t^{2n}\right)(K'(-t))^{2}\,dt \leq S_{r} \leq 2\pi \int_{0}^{r}t\left(1-|\zeta|^{2}t^{2n}\right)(K'(t))^{2}\, dt.
\end{equation}
\end{prop}

\begin{proof}
	Let $f=h+\overline{g} \in \mathcal{HC}_n(\phi)$. Then, the area of image of $\mathbb{D}_{r}$ under a harmonic mapping $f$ is given by
	\begin{equation} \label{eq-4.07}
		S_{r}=\iint_{\mathbb{D}_{r}} \left(|h'(z)|^{2}-|g'(z)|^{2}\right)\, dx dy= \iint_{\mathbb{D}_{r}} \left(1-|\zeta|^{2}|z|^{2n}\right)|h'(z)|^{2}dxdy.
	\end{equation}	
	Since $h\in \mathcal{C}(\phi)$, in view of \eqref{eq-4.01} and \eqref{eq-4.07}, we have
	\begin{equation} \label{410}
	\int_{0}^{r} \int_{0}^{2\pi}t\left(1-|\alpha|^{2}t^{2}\right)(K'(-t))^{2}d \theta dt \leq S_{r} \leq  \int_{0}^{r}\int_{0}^{2\pi}t\left(1-|\alpha|^{2}t^{2}\right)(K'(t))^{2}d\theta dt.
	\end{equation}	
Therefore, the assertion \eqref{eq-3.02} of Proposition \ref{prop-4.02} follows directly from \eqref{410}.
\end{proof}

In what follows, we derive the Bohr inequality for the class $\mathcal{HC}_n(\phi)$.

\begin{thm} \label{thm-3.04}
Let $f \in \mathcal{HC}_n(\phi)$. Then the majorant series of $f$ satisfies the inequality
\begin{equation} \label{eq-3.03}
|z|+\sum_{n=2}^{\infty} (|a_{n}|+|b_{n}|)|z|^{n} \leq d(f(0),\partial f(\mathbb{D}))
\end{equation}
for $|z|=r\leq \min \{1/3,r_{f}\}$, where $r_{f}$ is the smallest positive root in $(0,1)$ of $$L(\zeta,n,1)=M_{K}(r)+|\zeta|\int_{0}^{r} t^n M_{K'}(t)\,dt,$$
and $L(\zeta,n,1)$ is given by \eqref{41} with $r=1$.
\end{thm}

\begin{proof}
	Let $f=h+\overline{g} \in \mathcal{HC}_n(\phi)$. Since $h \in \mathcal{C}(\phi)$, from Lemma \ref{lem-2.04}, we know that
	\begin{equation} \label{eq-4.08}
		h' \prec K'.
	\end{equation}
	Let $K(z)=z+\sum \limits_{n=2}^{\infty} k_{n}z^{n}$.
	In view of Lemma \ref{lem-2.07} and \eqref{eq-4.08}, we have
	\begin{equation}\label{eq-4.09}
		1+\sum\limits_{n=2}^{\infty} n|a_{n}|r^{n-1}=M_{h'}(r) \leq M_{K'}(r)=1+\sum\limits_{n=2}^{\infty} n|k_{n}|r^{n-1}
	\end{equation}
	for $|z|=r\leq 1/3$. By integrating \eqref{eq-4.09} with respect to $r$ from $0$ to $r$, we get
	\begin{equation} \label{eq-4.10}
		M_{h}(r)=r+\sum\limits_{n=2}^{\infty} |a_{n}|r^{n} \leq r+\sum\limits_{n=2}^{\infty} |k_{n}|r^{n}=M_{K}(r)\quad (r\leq 1/3).
	\end{equation}
	From the definition of $\mathcal{HC}_n(\phi)$, we know that $$g'(z)=\zeta z^n h'(z).$$ This relationship along with \eqref{eq-4.09} yields
	\begin{equation} \label{eq-4.11}
		\sum\limits_{n=2}^{\infty} n|b_{n}|r^{n-1}=M_{g'}(r)=|\zeta|r^n M_{h'}(r) \leq |\zeta|r^n M_{K'}(r)\quad (r\leq 1/3).
	\end{equation}
	By integrating \eqref{eq-4.11} with respect to $r$ from $0$ to $r$, it follows  that
	\begin{equation} \label{eq-4.12}
		M_{g}(r)=\sum\limits_{n=2}^{\infty} |b_{n}|r^{n} \leq |\zeta| \int_{0}^{r} t^n M_{K'}(t)dt\quad (r\leq 1/3).
	\end{equation}
	Therefore, for $|z|=r\leq 1/3$, from \eqref{eq-4.10} and \eqref{eq-4.12},  we obtain
	\begin{equation} \label{eq-4.13}
		M_{f}(r)=|z|+\sum_{n=2}^{\infty} (|a_{n}|+|b_{n}|)r^{n} \leq M_{K}(r)+ |\zeta| \int_{0}^{r} t^n M_{K'}(t)dt=R_{\mathcal{C}}(n,r).
	\end{equation}
In view of \eqref{eq-3.01}, it is evident that the Euclidean distance between $f(0)$ and the boundary of $f(\mathbb{D})$ is given by
	\begin{equation} \label{eq-4.14}
		d(f(0), \partial f(\mathbb{D}))= \liminf \limits_{|z|\rightarrow 1} |f(z)-f(0)| \geq L(\zeta,n,1).
	\end{equation}
	We note that $R_{\mathcal{C}}(n,r) \leq L(\zeta,n,1)$ whenever $r \leq r_{f}$, where $r_{f}$ is the smallest positive root of $R_{\mathcal{C}}(n,r)=L(\zeta,n,1)$ in $(0,1)$. Let $$H_{1}(n,r)=R_{\mathcal{C}}(n,r)-L(\zeta,n,1).$$ Then $H_{1}(n,r)$ is a continuous function in $[0,1]$. Since $$M_{K}(r) \geq K(r)>-K(-r),$$ it follows that
\begin{equation}\label{eq-4.15}
		\begin{aligned}
		H_{1}(n,1)&=R_{\mathcal{C}}(n,1)-L(\zeta,n,1)\\
		&=  M_{K}(1)+K(-1)+|\zeta|\int_{0}^{r} t^n \left(M_{K'}(t)+K'(t)\right)dt \\
		&\geq  K(1)+K(-1)+|\zeta|\int_{0}^{r} t^n \left(M_{K'}(t)+K'(t)\right)dt>0.
	\end{aligned}
\end{equation}
	On the other hand,
	\begin{equation} \label{eq-4.16}
		H_{1}(n,0)=-L(\zeta,n,1)=K(-1)(1-|\zeta|)+n|\zeta|\int_{0}^{1}t^{n-1}K(-t)\,  dt<0.
	\end{equation}
	Therefore, $H_{1}$ has a root in $(0,1)$. Let $r_{f}$ be the smallest root of $H_{1}$ in $(0,1)$. Then $R_{\mathcal{C}}(n,r)\leq L(\zeta,n,1)$ for $r\leq r_{f}$.	Now, in view of the inequalities \eqref{eq-4.13} and \eqref{eq-4.14} with the relationship $R_{\mathcal{C}}(n,r)\leq L(\zeta,n,1)$ for $r\leq r_{f}$, we obtain
	$$
	|z|+\sum_{n=2}^{\infty} (|a_{n}|+|b_{n}|)r^{n} \leq d(f(0), \partial f(\mathbb{D}))
	$$
	for $|z|=r\leq \min \{1/3, r_{f}\}$.
\end{proof}

For a particular choice of $\phi$ in Theorem \ref{thm-3.04}, we get the following result.

\begin{cor} \label{thm-3.04a}
	Let $f \in \mathcal{M}(\alpha,\zeta,n)$  with $0\leq\alpha< 1$ and $0 \leq \zeta <1/(2n-1)$. Then the inequality \eqref{eq-3.03} holds for $|z|=r \leq r_{f}$, where $r_{f}$ is the smallest root in $(0,1)$ of $$F_{n}(r):=R(\alpha,\zeta,n,r)-L(\alpha,\zeta,n,1)=0.$$ The radius $r_{f}$ is sharp.
\end{cor}

\begin{proof}
	From Lemma \ref{lem-2.02},
	the Euclidean distance between $f(0)$ and the boundary of $f(\mathbb{D})$ shows that
	\begin{equation}\label{eq-4.17}
		d(f(0), \partial f(\mathbb{D}))= \liminf \limits_{|z|\rightarrow 1} |f(z)-f(0)|\geq L(\alpha,\zeta,n,1).
	\end{equation}
	We note that $r_{f}$ is the root of the equation $R(\alpha,\zeta,n,r)=L(\alpha,\zeta,n,1)$ in $(0,1)$. The existence of the root is ensured by the relation $R(\alpha,\zeta,n,1) >L(\alpha,\zeta,n,1)$ with \eqref{eq-2.05}. For $0<r\leq r_{f}$, it is evident that $R(\alpha,\zeta,n,r)\leq L(\alpha,\zeta,n,1)$. In view of Lemma \ref{lem-2.01} and \eqref{eq-4.17}, for $|z|=r\leq r_{f}$, we have
	\begin{align*}
		|z|+\sum\limits_{n=2}^{\infty} (|a_{n}|+|b_{n}|)|z|^{n}
		&\leq r_{f} + (|a_{2}|+|b_{2}|)r_{f}^{2}+\sum\limits_{n=3}^{\infty} (|a_{n}|+|b_{n}|)r_{f}^{n}\\
		& =R(\alpha,\zeta,n,r_{f}) \leq  L(\alpha,\zeta,n,1) \leq d(f(0), \partial f(\mathbb{D})).
	\end{align*}
	To show the sharpness of the radius $r_{f}$, we consider the function $f=f_{\alpha,\zeta,n}$, which is defined in Lemma \ref{lem-2.02}. We see that $f_{\alpha,\zeta,n}$ belongs to $\mathcal{M}(\alpha,\zeta,n)$. Since the left side of the growth inequality in Lemma \ref{lem-2.02} holds for $f=f_{\alpha,\zeta,n}$ or its rotations, we have $d(f(0), \partial f(\mathbb{D}))=L(\alpha,\zeta,n,1)$. Therefore, the function $f=f_{\alpha,\zeta,n}$ for $|z|=r_{f}$ gives
	\begin{align*}
		|z|+\sum\limits_{n=2}^{\infty} (|a_{n}|+|b_{n}|)|z|^{n}
		&= r_{f} + (|a_{2}|+|b_{2}|)r_{f}^{2}+\sum\limits_{n=3}^{\infty} (|a_{n}|+|b_{n}|)r_{f}^{n}\\
		& =R(\alpha,\zeta,n,r_{f}) = L(\alpha,\zeta,n,1) =d(f(0), \partial f(\mathbb{D})),
	\end{align*}
	which reveals that the radius $r_{f}$ is the best possible. 	
\end{proof}

The roots $r_{f}$ of $F_{n}(r)=0$ for different values of $\alpha$, $\zeta$ and $n$ have been shown in Table \ref{ftab-1}, Table \ref{ftab-2}    
and Figure \ref{figure-1}, respectively.

\begin{figure}[!htb]
	\begin{center}
		\includegraphics[width=0.45\linewidth]{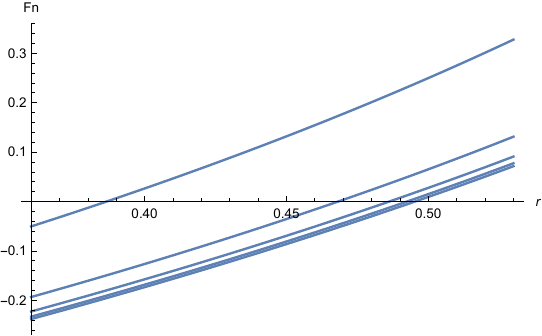}
	\includegraphics[width=0.45\linewidth]{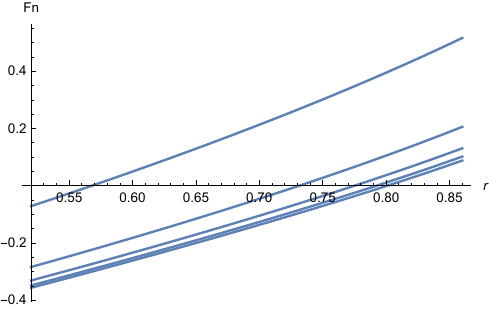}
	\end{center}
	\caption{The graphs of $F_{n}(r)$, respectively, for $\alpha=0.5$ and $\alpha=0.9$ when $n=1,2,3,4,5$.}
	\label{figure-1}
\end{figure}

\begin{table}[ht]
	\begin{tabular}{|l|l|l|l|l|l|l|l|l|}
		\hline
		$n$ & $1$ & $2$ & $3$ & $4$ & $5$ & $10$ & $100$& $1000$  \\
		\hline
		$\zeta$ & $1/2$ & $1/4$ & $1/6$& $1/8$ & $1/10$ & $1/20$ & $1/200$ & $1/2000$   \\
		\hline
		$r_{f}$ & {\small $0.386555$} & {\small $0.468176$} & {\small$0.486196$} & {\small $0.492459$} & {\small $0.495252$} & {\small $0.498809$} & {\small $0.499988$} & {\small $0.500000$} \\
\hline
	\end{tabular}
	\vspace{1mm}	
	\caption{The roots $r_{f}$ of $F_{n}(r)=0$ for different values of $\zeta$ when $\alpha=0.5$.}
	\label{ftab-1}
\end{table}

\begin{table}[ht]\label{tab2}
	\begin{tabular}{|l|l|l|l|l|l|l|l|l|}
			\hline
		$n$ & $1$ & $2$ & $3$ & $4$ & $5$ & $10$ & $100$& $1000$  \\
		\hline
		$\zeta$ & $1/2$ & $1/4$ & $1/6$& $1/8$ & $1/10$ & $1/20$ & $1/200$ & $1/2000$   \\
		\hline
		$r_{f}$ & {\small $0.567721$} & {\small $0.731273$} & {\small $0.774894$} & {\small $0.792253$ } & {\small $0.800709$} & {\small $0.812036$} & {\small $0.815292$} & {\small $0.815323$} \\
		\hline
	\end{tabular}
	\vspace{1mm}	
	\caption{The roots $r_{f}$ of $F_{n}(r)=0$ for different values of $\zeta$ when $\alpha=0.9$.}
	\label{ftab-2}
\end{table}

\begin{rem}{\rm
For $n=1$ and $\abs{\zeta}\leq1$, $r_{f}$ can be found in \cite{Allu-Halder-2020-2}.
For $\alpha=0.5$, when $n\rightarrow\infty$, the sharp radius is 0.500000.
For $\alpha=0.9$, when $n\rightarrow\infty$, the sharp radius is 0.815323.
For  $n=1$, when $\alpha\rightarrow1$, the sharp radius is 0.645751.
}
\end{rem}

Now, we give an improved version of Bohr inequality for the class $\mathcal{HC}_n(\phi)$. By adding area
quantity ${S_{r}}/{(2 \pi)}$ with the majorant series of $f \in \mathcal{HC}_n(\phi)$, the sum is still less
than $d(f(0),\partial f(\mathbb{D}))$ for some radius $r \leq \min \{1/3,\widetilde{r}_{f}\}<1$.

Note that the additional term such as ${S_{r}}/{(2 \pi)}$ to the majorant sum was first mooted by Kayumov and Ponnusamy \cite{Kayumov-Ponnusamy-2020} to refine and improve the Bohr inequality. This variation of Bohr inequality was proved for harmonic mappings in  \cite{Evdoridis-Ponnusamy-Rasila-2019}. Subsequently, several extensions were made by many authors (cf. \cite{Liu-Liu-Ponnusamy-2021}).

\begin{thm} \label{thm-3.05}
	Let $f \in \mathcal{HC}_n(\phi)$ and $S_{r}$ be the area of the image $f(\mathbb{D}_{r})$. Then the inequality
	$$
	M_{f}(r)+ \frac{S_{r}}{2\pi} \leq d(f(0),\partial f(\mathbb{D}))
	$$
holds for $|z|=r\leq  \min \{1/3,\widetilde{r}_{f}\}$, where $\widetilde{r}_{f}$ is the smallest positive root in $(0,1)$  of $$L(\zeta,n,1)
	=M_{K}(r)+|\zeta|\int_{0}^{r} t^n M_{K'}(t)\, dt + \int_{0}^{r}t\left(1-|\zeta|^{2}t^{2n}\right)(K'(t))^{2} dt,$$ and $L(\zeta,n,1)$ is given by \eqref{41} with $r=1$.
\end{thm}

\begin{proof}
	Let $f \in \mathcal{HC}_n(\phi)$ be of the form \eqref{eq-1.01}. Then, from the right hand inequality in \eqref{eq-3.02} and \eqref{eq-4.13}, we obtain
\begin{equation}\label{eq-4.18}
	\begin{aligned}
		M_{f}(r)+ \frac{S_{r}}{2\pi} &\leq M_{K}(r)+|\zeta|\int_{0}^{r} t^n M_{K'}(t)\,  dt + \int_{0}^{r}t\left(1-|\zeta|^{2}t^{2n}\right)(K'(t))^{2}\, dt \\
		&= R_{\mathcal{C}}(n,r) +\int_{0}^{r}t\left(1-|\zeta|^{2n}t^{2}\right)(K'(t))^{2}\,  dt =\widetilde{R}_{f}(n,r)
	\end{aligned}
\end{equation}
	for $r\leq 1/3$. Suppose that $H_{2}(n, r)=\widetilde{R}_{f}(n,r)-L(\zeta,n,1)$. Then $H_{2}(n,r)$ is a continuous function in $[0,1]$. The inequality \eqref{eq-4.16} yields that $H_{2}(n,0)=-L(\zeta,n,1)<0$. By virtue of  \eqref{eq-4.15}, we get
	\begin{equation}\label{eq-4.19}
		R_{\mathcal{C}}(n,1)-L(\zeta,n,1)>0.
	\end{equation}
For $|\zeta|\leq1/(2n-1)$, we observe that $$t\left(1-|\zeta|^{2}t^{2n}\right)(K'(t))^{2}\geq 0,$$ and hence
	\begin{equation}\label{eq-4.20}
		\int_{0}^{r}t\left(1-|\zeta|^{2}t^{2n}\right)(K'(t))^{2}dt\geq0.
	\end{equation}
	From \eqref{eq-4.18} and \eqref{eq-4.19}, we obtain
	$$
	H_{2}(n,1)=R_{\mathcal{C}}(n,1)-L(\zeta,n,1) + \int_{0}^{1}t\left(1-|\zeta|^{2}t^{2n}\right)(K'(t))^{2}\,  dt>0.
	$$
	Since $H_{2}(n,0)<0$ and $H_{2}(n,1)>0$, $H_{2}$ has a root in $(0,1)$ and choose $\widetilde{r}_{f}$ to be the smallest root in $(0,1)$, we know that $\widetilde{R}_{f}(n,r)\leq L(\zeta,n,1)$ for $r \leq \widetilde{r}_{f}$. Therefore, by virtue of \eqref{eq-4.14} and \eqref{eq-4.18}, we conclude that
	$$
	M_{f}(r)+ \frac{S_{r}}{2\pi}\leq d(f(0), \partial f(\mathbb{D}))
	$$
	for $r \leq \min \{1/3, \widetilde{r}_{f}\}$.		
\end{proof}

\begin{rem}{\rm
By setting $n=1$ in Theorems \ref{thm-3.04} and \ref{thm-3.05}, we get the corresponding results obtained in \cite{Allu-Halder-2020-2}.}
\end{rem}

\vskip .20in
\begin{center}{\sc Acknowledgments}
\end{center}

\vskip.01in
The present investigation was supported by the \textit{Key Project of Education Department of Hunan Province} under Grant no.
19A097, and the \textit{Natural Science Foundation of Hunan Province} under Grant no. 2018JJ2074 of the P. R. China. The authors would like to thank the referees for their valuable comments and suggestions, which was essential to improve the quality of this paper.

\vskip .20in
\begin{center}{\sc Data availability statement}
\end{center}
\vskip.01in

No data, models, or code were generated or used during the study (e.g. opinion
or dateless paper).

\vskip .20in
\begin{center}{\sc Conflict of interest}
\end{center}
\vskip.01in

The authors declare that they have no conflict of interest.







\vskip.20in

\begin{thebibliography}{99}

\bibitem{Ahamed-Allu-2021}
M. B. Ahamed, V. Allu and H. Halder, Bohr radius for certain classes of close-to-convex harmonic mappings,
{\it Anal. Math. Phys.}  {\bf  11}  (2021),  1--30.
	
\vskip.05in
\bibitem{Ahuja-Khatter-2021}
O. P. Ahuja, K. Khatter and  V. Ravichandran,  Toeplitz determinants associated with Ma-Minda classes of starlike
	and convex functions, {\it Iran. J. Sci. Technol. Trans. Sci.}  \textbf{45} (2021), 2021--2027.

\vskip.05in
\bibitem{Aizenberg-2000}
L. Aizenberg, Multidimensional analogues of Bohr's theorem on power series,
{\it Proc. Amer. Math. Soc.} {\bf 128} (2000), 1147--1155.
	
\vskip.05in
\bibitem{Aizenberg-Aytuna-Djakov-2000}
L. Aizenberg, A. Aytuna  and  P. Djakov,  An abstract approach to Bohr  phenomenon,
{\it Proc. Amer. Math. Soc.} {\bf 128} (2000), 2611--2619.
	
\vskip.05in
\bibitem{Alkhaleefah-Kayumov-Ponnusamy-2019}
S. Alkhaleefah, I. Kayumov  and S. Ponnusamy,  On the Bohr inequality with a fixed zero coefficient,
{\it Proc. Amer. Math. Soc.} {\bf 147}  (2019),  5263--5274.

\vskip.05in
\bibitem{Ali-Abdulhadi-Ng-2016}
 R. M. Ali,  Z. Abdulhadi and  Z. C. Ng, The Bohr radius for starlike logharmonic mappings,
 {\it Complex Var. Elliptic Equ.}  {\bf 61} (2016), 1--14.

 \vskip.05in
 \bibitem{Ali-Thomas-Vasudevarao-2018}
 M. F. Ali, D. K. Thomas and A. Vasudevarao,  Toeplitz determinants whose elements are the coefficients of
 analytic and univalent functions, {\it Bull. Aust. Math. Soc.}  {\bf  97} (2018), 253--264.

\vskip.05in
\bibitem{Allu-Halder-2021}
V. Allu and  H. Halder,  Bohr radius for certain classes of starlike and convex univalent functions,
{\it J. Math. Anal. Appl.}   {\bf 493} (2021), Art. 124519, 15 pp.

\vskip.05in
\bibitem{Allu-Halder-2020-2}
V. Allu and  H. Halder, The Bohr inequality for certain harmonic mappings, \textit{Indag. Math. (N.S.)} (2021),
doi: 10.1016/j.indag.2021.12.004.

\vskip.05in
\bibitem{Arif-Raza-2019}
M. Arif, M. Raza, H. Tang, S. Hussain and H. Khan, Hankel determinant of order three for
	familiar subsets of analytic functions related with sine function, {\it Open Math.} {\bf 17} (2019), 1615--1630.

\vskip.05in
\bibitem{Babalola-2010}
K. O. Babalola, On $H_3(1)$ Hankel determinant for some classes of univalent functions,
a book chapter in ``Inequality Theory and Applications", Eds: Y. J. Cho, J. K. Kim and S. S. Dragomir,  7 pages. Nova Science Publishers Inc.  2011.

\bibitem{Bharanedhar-Ponnusamy-2014}
S. V. Bharanedhar and S. Ponnusamy, Coefficient conditions for harmonic univalent mappings
and hypergeometric mappings, {\it Rocky Mountain J. Math.} {\bf 44} (2014), 753--777.


\vskip.05in
\bibitem{Bhowmik-Das-2018}
B. Bhowmik and  N. Das,  Bohr phenomenon for subordinating families of certain univalent functions,
{\it J. Math. Anal. Appl.}  {\bf 462} (2018), 1087--1098.

\vskip.05in
\bibitem{Blasco-2010}
 O. Blasco, The Bohr radius of a Banach space. Vector measures, integration and related topics, 59--64, Oper. Theory Adv. Appl., 201, Birkh\"{a}user Verlag, Basel, 2010.

\vskip.05in
\bibitem{Boas--Khavinson-1997}
 H. P. Boas  and   D. Khavinson, Bohr's power series theorem in several variables, {\it Proc. Amer. Math. Soc.}  {\bf 125} (1997), 2975--2979.

\vskip.05in
\bibitem{Bohr-1914}
H. Bohr, A theorem concerning power series, {\it Proc. Lond. Math. Soc.} {\bf 2} (1914), 1--5.

\vskip.05in
\bibitem{Cho-Kumar-2020}
N. E. Cho, S. Kumar and V. Kumar,  Hermitian-Toeplitz and Hankel determinants for certain starlike functions,
{\it Asian-European J. Math.}  (2021), doi: 10.1142/S1793557122500425.

\vskip.05in
\bibitem{Clunie-Sheil-Small-1984}
J. Clunie and T. Sheil-Small, Harmonic univalent functions, {\it  Ann. Acad. Sci. Fenn. Ser. A. I. Math.} {\bf 9} (1984), 3--25.

\vskip.05in
\bibitem{Cudna-Kwon-2019}
K. Cudna,  O. S. Kwon, A. Lecko,   Y. J. Sim and B. \'{S}miarowska, The second and third-order Hermitian Toeplitz determinants for
starlike and convex functions of
order $\alpha$,  {\it Bol. Soc. Mat. Mex.}  {\bf 26}  (2020),  361--375.

\vskip.05in
\bibitem{Dobosz-2021}
A.  Dobosz, The third-order Hermitian Toeplitz determinant for alpha-convex functions, {\it  Symmetry}  {\bf 13} (2021), Art. 1274, 7 pp.

\vskip.05in
\bibitem{Dixon-1995}
P. G. Dixon, Banach algebras satisfying the non-unital von Neumann inequality, {\it Bull. London Math. Soc.} {\bf  27} (1995), 359--362.

\vskip.05in
\bibitem{Duren-2004}
 P. Duren, Harmonic mappings in the plane,  Cambridge Univ. Press, 2004.

\vskip.05in
\bibitem{Duren-1983}
 P. Duren,  Univalent functions, Springer-Verlag, 1983.

\vskip.05in
\bibitem{Efraimidis-2016}
I. Efraimidis,  A generalization of Livingston's coefficient inequalities for functions with positive
 real part, {\it J. Math. Anal. Appl.} {\bf 435}  (2016), 369--379.

\vskip.05in
\bibitem{Evdoridis-Ponnusamy-Rasila-2019}
S. Evdoridis,  S. Ponnusamy and  A. Rasila, Improved Bohr's inequality for locally univalent harmonic mappings,
{\it Indag. Math. (N.S.)} {\bf 30}   (2019), 201--213.

\vskip.05in
\bibitem{Huang-Liu-Ponnusamy-2021}
Y. Huang, M.-S. Liu and  S. Ponnusamy, Bohr-Type inequalities for harmonic mappings with a multiple zero
at the origin,  {\it Mediterr. J. Math.}  {\bf 18}  (2021), 1--22.

\vskip.05in
\bibitem{Ismagilov-Kayumov-Ponnusamy-2020}
A. Ismagilov, I. R. Kayumov and  S. Ponnusamy, Sharp Bohr type inequality, {\it J. Math. Anal. Appl.}  {\bf 489} (2020), Art. 124147, 10 pp.

\vskip.05in
\bibitem{Jastrzebski-Kowalczyk-2020}
P. Jastrz\c{e}bski, B. Kowalczyk, O. S. Kwon and A. Lecko, Hermitian Toeplitz determinants of the second and
third-order for classes of close-to-star functions, {\it RACSAM Rev. R. Acad. A.} {\bf 114} (2020), 1--14.

\vskip.05in
\bibitem{Janteng-Halim-Darus-2007}
 A. Janteng, S. A. Halim  and   M. Darus, Hankel determinant for starlike and convex functions,
 {\it Int. J. Math. Anal.} {\bf 1} (2007), 619--625.

\vskip.05in
\bibitem{Kayumov-Ponnusamy-2018}
I. R. Kayumov and  S. Ponnusamy, Bohr's inequalities for the analytic functions with lacunary series and
harmonic functions, {\it J. Math. Anal. Appl.}  {\bf 465} (2018), 857--871.

 \vskip.05in
\bibitem{Kayumov-Ponnusamy-2020}
I. R. Kayumov and  S. Ponnusamy, Improved version of Bohr's inequalities,
{\it C. R. Math. Acad. Sci. Paris} {\bf 358} (2020), 615--620.


\vskip.05in
\bibitem{Kayumov-Ponnusamy-Shakirov-2018}
I. R. Kayumov, S. Ponnusamy  and N. Shakirov, Bohr radius for locally univalent harmonic mappings,
{\it Math. Nachr.} {\bf 291} (2018), 1751--1768.

\vskip.05in
\bibitem{Kowalczyk-Kwon-2019}
B. Kowalczyk, O. S. Kwon, A. Lecko, Y. J. Sim and B. \'{S}miarowska, The third-order Hermitian Toeplitz determinant for classes of
functions convex in one direction, {\it Bull. Malays. Math. Sci. Soc.} {\bf 43}   (2020), 3143--3158.

\vskip.05in
\bibitem{Kumar-2021}
V. Kumar  and S. Kumar,  Bounds on Hermitian-Toeplitz and Hankel determinants for strongly starlike functions,
{\it Bol. Soc. Mat. Mex.}   {\bf  27} (2021),  1--16.

\vskip.05in
\bibitem{Lecko-miarowska-2021}
A. Lecko  and B. \'{S}miarowska,  Sharp bounds of the Hermitian Toeplitz determinants for some classes of close-to-convex
functions, {\it Bull. Malays. Math. Sci. Soc.} {\bf  44}  (2021), 3391--3412.

\vskip.05in
\bibitem{Lecko-Sim-2020}
A. Lecko, Y. J. Sim and B.  \'{S}miarowska, The fourth-order Hermitian Toeplitz determinant for convex functions,
{\it Anal. Math. Phys.}  {\bf 10}  (2020),  1--11.

\vskip.05in
\bibitem{Lewy-1936}
H. Lewy, On the non-vanishing of the Jacobian in certain one-to-one mappings,
{\it Bull. Amer. Math. Soc.}   {\bf 42} (1936), 689--692.

\vskip.05in
\bibitem{Liu-2021}
G. Liu,  Bohr-type inequality via proper combination,
{\it J. Math. Anal. Appl.}  {\bf 503} (2021), Art. 125308, 17 pp.

\bibitem{Liu-Liu-Ponnusamy-2021}
G. Liu, Z.-H. Liu, and S. Ponnusamy, Refined Bohr inequality for bounded analytic functions, {\it Bull.
Sci. Math.} {\bf 173} (2021),  Art. 103054, 20 pp.


\vskip.05in
\bibitem{Liu-Ponnusamy-2021}
M.-S. Liu and  S. Ponnusamy, Multidimensional analogues of refined Bohr's inequality,
{\it Proc. Amer. Math. Soc.}  {\bf 149} (2021), 2133--2146.

\vskip.05in
\bibitem{Liu-Ponnusamy-2019}
Z.-H. Liu and  S. Ponnusamy, Bohr radius for subordination and $k$-quasiconformal harmonic mappings,
{\it Bull. Malays. Math. Sci. Soc.}  {\bf 42} (2019),  2151--2168.

\vskip.05in
\bibitem{Ma-Minda-1992}
 W. C. Ma and D. Minda, A unified treatment of some special classes of univalent functions, in
 {\it Proceedings of the Conference on Complex Analysis} (Tianjin, 1992), 157--169,
 Conf. Proc. Lecture Notes Anal. I, Int. Press, Cambridge.

\vskip.05in
\bibitem{Muhanna-Ali-2014}
Y. A. Muhanna, R. M. Ali, Z. C. Ng and S. F. M. Hasni, Bohr radius for subordinating families of analytic
functions and bounded harmonic mappings, {\it J. Math. Anal. Appl.} {\bf 420} (2014), 124--136.

\vskip.05in
\bibitem{Muhanna-2010}
Y. A. Muhanna,  Bohr's phenomenon in subordination and bounded harmonic classes,
{\it Complex Var. Elliptic Equ.} {\bf  55} (2010), 1071--1078.

\vskip.05in	
\bibitem{Paulsen-Singh-2004}
V. I. Paulsen  and   D. Singh, Bohr  inequality for uniform algebras, {\it Proc. Amer. Math. Soc.} {\bf 132} (2004), 3577--3579.

\vskip.05in
\bibitem{Ponnusamy-Rasila-2013}
S. Ponnusamy and A. Rasila, Planar harmonic and quasiregular mappings, in: Topics in Modern Function Theory:
Chapter in CMFT, in: RMS-Lecture Notes Series.  {\bf 19} (2013),  267--333.

\vskip.05in
\bibitem{Ponnusamy-Kaliraj-2015}
S. Ponnusamy and A. S. Kaliraj, Constants and characterization for certain classes of univalent harmonic mappings,
 {\it Mediterr. J. Math.} {\bf 12} (2015), 647--665.

\bibitem{Ponnusamy-Vijayakumar-Wirths-2022}
S. Ponnusamy, R. Vijayakumar and K.-J. Wirths, Improved Bohr's phenomenon in quasi-subordination classes, {\it J. Math. Anal.   Appl.} {\bf 506}  (2022), Art. 125645, 10 pp.

\vskip.05in
 \bibitem{Radhika-Sivasubramanian-Murugusundaramoorthy-Jahangiri-2016}
 V. Radhika, S. Sivasubramanian, G. Murugusundaramoorthy and   J. M. Jahangiri,
Toeplitz matrices whose elements are the coefficients of functions with bounded boundary rotation,
{\it J. Complex Anal.}  \textbf{2016} (2016), Art. ID 4960704, 4 pp.

\vskip.05in
\bibitem{Sharma-Raina-2019}
P. Sharma, R. K.  Raina   and J. Sok\'{o}{\l}, Certain Ma-Minda type classes of analytic functions associated with
the crescent-shaped region, {\it  Anal.  Math. Phys.} {\bf 9} (2019),  1887--1903.


\vskip.05in
\bibitem{Sun-Wang-2019}
Y. Sun, Z.-G. Wang and A. Rasila,  On third Hankel determinants for subclasses of analytic functions and close-to-convex
harmonic mappings, {\it Hacet. J. Math. Stat.} {\bf 48}  (2019),  1695--1705.

\vskip.05in
\bibitem{Sun-Jiang-Rasila-2016}
Y. Sun,  Y.-P. Jiang and  A. Rasila, On a certain subclass of close-to-convex harmonic mappings,
{\it Complex Var. Elliptic Equ.} {\bf 61} (2016), 1627--1643.

\vskip.05in
\bibitem{Thomas-Tuneski-Vasudevarao-2018}
 D. K. Thomas, N. Tuneski and  A. Vasudevarao,  Univalent functions: A primer. De Gruyter Studies in Mathematics,
69. De Gruyter, Berlin,  2018.


\vskip.05in
\bibitem{Wang-Huang-Long-2021}
D.-R. Wang, H.-Y. Huang and B.-Y.  Long, Coefficient problems for subclasses of close-to-star functions,
{\it Iran. J. Sci. Technol. Trans. Sci.} {\bf 45} (2021), 1071--1077.

\vskip.05in
\bibitem{whlk}
Z.-G. Wang, X.-Z. Huang, Z.-H. Liu and R. Kargar, On quasiconformal close-to-convex harmonic mappings involving starlike functions,
\textit{Acta Math. Sinica (Chin. Ser.)} {\bf 63} (2020), 565--576.

\vskip.05in
\bibitem{Wang-Liu-2018}
Z.-G. Wang, Z.-H. Liu, A. Rasila and Y. Sun, On a problem of Bharanedhar and Ponnusamy involving planar harmonic mappings,
{\it Rocky Mountain J. Math. } {\bf 48} (2018), 1345--1358.


\vskip.05in
\bibitem{Wani-Swaminathan-2021}
L. A. Wani  and A. Swaminathan, Starlike and convex functions associated with a nephroid domain,
{\it Bull. Malays. Math. Sci. Soc.} {\bf  44} (2021), 79--104.

\vskip.05in
\bibitem{Zhang-Srivastava-2019}
H.-Y. Zhang, R. Srivastava and H. Tang, Third-order Hankel and Toeplitz determinants for starlike functions
connected with the sine function, {\it  Mathematics} {\bf 7} (2019), Art. 404, 10 pp.



\end{thebibliography}
\end{document}